\documentclass[12pt,twoside,reqno]{amsart}
\usepackage{amsfonts, amsthm, amsmath, amssymb}
\usepackage{hyperref}
\hypersetup{colorlinks=false}

\usepackage[margin=1.5in]{geometry}

\RequirePackage{mathrsfs} \let\mathcal\mathscr
\numberwithin{equation}{section}

    \renewcommand{\phi}{\varphi}
    \renewcommand{\rho}{\varrho}

    \renewcommand{\leq}{\leqslant}
    
    \renewcommand{\geq}{\geqslant}
    \renewcommand{\ge}{\geqslant}

    \newcommand{\bea}{\begin{eqnarray}}
    \newcommand{\eea}{\end{eqnarray}}
    \newcommand{\bna}{\begin{eqnarray*}}
    \newcommand{\ena}{\end{eqnarray*}}

    \newtheorem{thm}{Theorem}[section]
    
    \newtheorem{lem}[thm]{Lemma}

    \theoremstyle{definition}

    \numberwithin{equation}{section}
\allowdisplaybreaks[3]

\begin{document}
\title
[{ On Waring--Goldbach problem with mixed powers}]
{ On Waring--Goldbach problem with mixed powers  }
\author
[Huimin Wang ]
{Huimin Wang }
\address{(Huimin Wang)}
\email{huimin97@126.com}
\address{ School of Mathematical Sciences,
Key Laboratory of Intelligent Computing and Applications (Ministry of Education), Tongji University, Shanghai 200092, P. R. China}

\author
[Shuangrui Tian*]
{Shuangrui Tian*}
\address{(Shuangrui Tian)}
\email{tianshuangrui@tongji.edu.cn}
\address{ School of Mathematical Sciences,
Key Laboratory of Intelligent Computing and Applications (Ministry of Education), Tongji University, Shanghai 200092, P. R. China}

\thanks{{\em Key words and phrases:} Waring--Goldbach problem, circle method, additive representations}
\thanks{2010  {\em Mathematics Subject Classification:} 11P32, 11P55}
\thanks{* Corresponding author.}

\begin{abstract}

It is proved that for every sufficiently large odd integer $n$, the equation
$$ n=p_1^2+p_2^2+p_3^3+p_4^4+p_5^5+p_6^6+p_7^8 $$
is solvable in primes. This refines the work of Cai and Mu in 2015.
\end{abstract}
\maketitle
\section{ Introduction}
The philosophy of the Hardy--Littlewood circle method suggests that one may
expect to solve equations of the form
\begin{equation}\label{eq:1.1}
x_1^{k_1}+x_2^{k_2}+\cdots+x_s^{k_s}=n
\end{equation}
in natural numbers $x_j\ (1\leq j\leq s)$, where the exponents
$k_j\ (1\leq j\leq s)$ are fixed positive integers satisfying
\[
k_1^{-1}+\cdots+k_s^{-1}>1,
\]
provided that $n$ is sufficiently large and satisfies the necessary
congruence conditions. One line of research in this direction concerns
representations involving two squares together with several additional
higher powers.

Hooley~\cite{Hooley} introduced divisor sum techniques into the investigation
of \eqref{eq:1.1}. In particular, he proved that for every sufficiently large
positive integer $n$, the equation
\begin{equation*}
x_1^2+x_2^2+x_3^3+x_4^4+x_5^5+x_6^6+x_7^7=n
\end{equation*}
has solutions in non-negative integers.

Moreover, Cai and Mu \cite{CaiMu} refined Hooley's results. Let $P_r$ denote an almost-prime
with at most $r$ prime factors, counted with multiplicity. They showed that the
equations
\begin{align}\label{woai}
w_1^2+p_2^2+p_3^3+p_4^4+p_5^5+p_6^6+p_7^6 &= n
\end{align}
\begin{align}\label{1.2}
y_1^2+p_2^2+p_3^3+p_4^4+p_5^5+p_6^6+p_7^7 &= n,
\end{align}
\begin{align}\label{1.3}
z_1^2+p_2^2+p_3^3+p_4^4+p_5^5+p_6^6+p_7^8 &= n
\end{align}
are solvable with $w_1 \in P_6$, $y_1 \in P_6$, and $z_1 \in P_7$, while all remaining variables are primes.
Recently, Br\"udern \cite{B} introduced new ideas and proved that \eqref{1.2}
has solutions in which all variables are primes. Subsequently, Geovane Matheus Lemes Andrade \cite{G} established that \eqref{woai} also has solutions with all variables restricted to primes.

It is therefore natural to ask whether \eqref{1.3} also admits solutions
with all variables restricted to primes. Compared with \eqref{woai} and
\eqref{1.2}, however, the presence of the eighth power makes the problem
substantially more delicate, since higher powers are generally harder to
treat within the Hardy--Littlewood circle method. In this paper, we
resolve this question by adapting Br\"udern's ideas \cite{B}.

\begin{thm}\label{wang}
 Let $v(n)$ denote the number of representations of the sufficiently large odd integer $n$ in the form
 $$  n=p_1^2+p_2^2+p_3^3+p_4^4+p_5^5+p_6^6+p_7^8.         $$
 Then $v(n)\gg n^\frac{39}{37}/\log^7 n$.
 \end{thm}
\noindent\textbf{Remark 1.}
The Hardy--Littlewood heuristic suggests that the expected lower bound is
\[
v(n)\gg \frac{n^{2\cdot\frac12+\frac13+\frac14+\frac15+\frac16+\frac18-1}}{\log^7 n}
=\frac{n^{43/40}}{\log^7 n}.
\]
The exponent $39/37$ in Theorem \ref{wang} is smaller because we restrict the fifth-power prime variable to the shortened range $p_5\leq n^{\gamma/5}$, where $\gamma=265/296$. This shortening is needed to obtain a sufficient saving on the minor arcs, ensuring that their contribution is smaller than the major-arc main term.

\par\medskip
\noindent\textbf{Remark 2.}
The eighth-power case presents an additional technical difficulty.
Replacing the seventh-power variable by an eighth-power variable shortens
the corresponding exponential sum from length $n^{1/7}$ to $n^{1/8}$.
Consequently, the mean-value estimates used in the preceding cases do not
directly provide sufficient control of the minor arcs. Thus, the
eighth-power case cannot be obtained as a straightforward extension of the
previous arguments and requires a different treatment of the minor-arc
contribution.

\par\medskip
\noindent\textbf{Remark 3.}
The main innovations of our argument are threefold. First, we restrict the
fifth-power prime variable to a suitably shortened range, which enables us
to balance the major-arc main term against the available minor-arc
estimates. Second, we establish the mixed mean-value estimates $J_8$ and
$J_6$ (see \eqref{J_8}) by applying Vaughan's iterative method to the
corresponding auxiliary Diophantine equations. These estimates provide the
key input for the minor-arc analysis. Third, we divide the minor arcs
according to the size of the cubic exponential sum. On the region where
this sum is small, H\"older's inequality together with the estimates for
$J_8$ and $J_6$ yields a power saving. On the complementary region,
rational approximation and estimates for exponential sums over primes
provide the required logarithmic saving.

\textbf{Notation.}
 As usual, let $\phi(q)$ denote Euler's function and $e(\alpha)=e^{2\pi i \alpha}$. The letter $p$ denotes a prime. We denote by $(a,b)$ the greatest common divisor of $a$ and $b$.
 Whenever $\varepsilon$ appears in a statement,
either implicitly or explicitly, we assert that the statement holds for each
$\varepsilon >0$. Note that the value of $\varepsilon$ may consequently change from statement
to statement. We assume that $n$ is  a large odd positive integer. Write
$  L=\log n.$

\section{Preliminaries}
Let $P_k=n^{1/k}.$ We define the exponential sum
$$
g_k(\alpha)=
\begin{cases}
\displaystyle \sum_{\frac12 P_k<p\leq P_k} e(\alpha p^k)\log p, & k\in\{2,3,4,6,8\},\\[8pt]
\displaystyle \sum_{p\leq P_5^{\gamma}} e(\alpha p^5)\log p, & k=5,
\end{cases}
$$
where $\gamma=265/296$.
Set
$$
G(\alpha)=g_2(\alpha)^2g_3(\alpha)g_4(\alpha)g_5(\alpha)g_6(\alpha)g_8(\alpha).
$$
By the orthogonality of additive characters, we have
\begin{align}\label{2.1}
    v(n)\gg \frac{1}{\log ^7n}\int_0^1 G(\alpha)e(-\alpha n)\,d\alpha.
\end{align}
Write
\begin{align*}
\Gamma=\frac13+\frac14+\frac{\gamma}{5}+\frac16+\frac18=\frac{39}{37}.
\end{align*}

Let $0 \leq a \leq q$ with $(a,q)=1$, and define $Q_1=n^{1/25}$ and $Q_2=n^{24/25}$. Put
\begin{align*}
\mathfrak{M}(q,a)=\left\{\alpha \in [0,1]:\left|\alpha-\frac{a}{q}\right|\leq\frac{L^{B}}{n}\right\},\quad   \mathfrak{M}=\bigcup_{1\leq q\leq L^{B}}\bigcup_{\substack{0\leq a\leq q\\(a,q)=1}}\mathfrak{M}(q,a),
\end{align*}

\begin{align*}
    \mathfrak{M}_2(q,a) = \left\{\alpha \in [0,1]:\ \frac{L^{B}}{n}\leq\left|\alpha-\frac{a}{q}\right|\leq \frac{1}{qQ_2}\right\},\quad \mathfrak{m_2}= \bigcup_{1\leq  q\leq L^{B}}\ \bigcup_{\substack{0\leq a\leq q\\ (a,q)=1}} \mathfrak{M}_2(q,a),
\end{align*}

\begin{align*}
    \mathfrak{M}_3(q,a) = \left\{\alpha \in [0,1]:\ \left|\alpha-\frac{a}{q}\right|\leq \frac{1}{qQ_2}\right\},\quad  \mathfrak{m_3} = \bigcup_{L^{B}<q\leq Q_1}\ \bigcup_{\substack{0\leq a\leq q\\ (a,q)=1}}
\mathfrak{M}_3(q,a),
\end{align*}
$$
 \mathfrak{n}
=
\left\{
\alpha\in[0,1]:
\left|g_3(\alpha)\right|
\leq
P_3^{\,1-\eta}
\right\}(\eta=11/200), \quad  \widetilde{\mathfrak{m}}=\mathfrak{m_2}\bigcup \mathfrak{m_3},
$$
where $B$ is a positive constant  to be chosen later.

\section{The Major arcs }
For $\alpha=\frac{a}{q}+\beta,$  define
$$
S_k(q,a)=\sum_{\substack{x=1\\(x,q)=1}}^{q} e\!\left(\frac{ax^k}{q}\right)\quad(2\leq k\leq 8),$$

$$v_k(\beta)=\frac1k\sum_{2^{-k}n<m\leq n} m^{1/k-1} e(\beta m), \quad (k=2,3,4,6,8),
$$
$$
U(q,a)=S_2(q,a)^2 S_3(q,a)S_4(q,a)S_5(q,a)S_6(q,a)S_8(q,a),
$$
$$
w(\beta)=v_2(\beta)^2 v_3(\beta)v_4(\beta)v_6(\beta)v_8(\beta),
$$
$$
A_n(q)=\sum_{\substack{a=1\\(a,q)=1}}^{q} U(q,a)e(-an/q).
$$

We state some preliminary lemmas which are required in Lemma \ref{Lemma 3.6}.
\begin{lem}\cite[Lemma 4]{Hua}\label{Lemma 3.1}
    Let $p^\theta \mid k,\; p^{\theta+1}\nmid k$, and
    $$l=
\begin{cases}
\theta+2, & p=2,\\
\theta+1, & p\ne 2.
\end{cases}$$
 If $t> l,$   then
    $S_k(p^t,a)=0.$
\end{lem}

\begin{lem}\label{Lemma 3.2}
 \item[(i)]  For $k\geq 2,$ we have
    $$ S_k(q,a)\ll q^{\frac{1}{2}+\varepsilon}. $$
    \item[(ii)]  For $(p,a)=1$, we have
$$
|S_k(p,a)| \leq \big((k,p-1)-1\big) p^{1/2}+1.
$$
\end{lem}
\begin{proof}
    For ($\mathrm{i}$) and ($\mathrm{ii}$), see \cite[Lemma 5]{Hua} and \cite[Lemma 4.3]{VaughanBook}, respectively.
\end{proof}

\begin{lem}\cite[Lemma 6.2]{VaughanBook}\label{Lemma 3.3}
    Suppose that $|\beta|\leq \frac{1}{2}. $ Then
    $$v_k(\beta)\ll n^{1/k}(1+n|\beta|)^{-1}.  $$
\end{lem}
\begin{lem}\label{Lemma 3.4}
   Suppose that $(q_1,q_2)=1.$
   Then $$A_n(q_1q_2)=A_n(q_1)A_n(q_2).$$
\end{lem}

\begin{proof}
    The proof is similar to \cite[Lemma 2.11]{VaughanBook}.
\end{proof}

\begin{lem}\label{Cauchy}
     Let $\mathcal{A},\mathcal{B}$ respectively denote sets of $r,s$ residue classes modulo $q.$
     Suppose further that $0\in \mathcal{B}$ and that for every $b \in \mathcal{B}$ with $b \not\equiv 0 \pmod{q}$ one has $(b,q)=1.$ Let
$\mathcal{A}+\mathcal{B}$ denote the set of residue classes modulo $q$ of the form $a+b$ with $a \in \mathcal{A}$ and $b \in \mathcal{B}.$ Then
$$ card(\mathcal{A}+\mathcal{B})\geq \min (q,r+s-1). $$
\end{lem}
\begin{proof}
    This is the well-known Cauchy--Davenport theorem. See \cite[Lemma 2.14]{VaughanBook}.

\end{proof}

\begin{lem}\label{Mn}

let $F_p$ denote the finite field of $p$ elements.
Let $M_n(p)$ denote the number of solutions of
$$
x_1^2 + x_2^2 + x_3^3 + x_4^4 + x_5^5 + x_6^6 + x_7^8 \equiv n \pmod{p}
$$
with all $x_i \in \mathbb{F}_p^*=F_p\setminus\{0\}.$  Then
 for all  positive odd  $n$ we have
$$
M_n(p)\ge 1
$$

\end{lem}

\begin{proof}
For sets $A,B$ in $\mathbb{F}_p$, write
$$
A+B=\{a+b:\ a\in A,\ b\in B\}.$$
The assertion is readily checked for all  positive odd  $n$ when $p=2,3,5,7$. We may therefore
suppose that $p\geq 11$. Let
$$
A_k=\{x^k \bmod p:x\in \mathbb F_p^\ast\}\subseteq \mathbb F_p
$$
and
$$
S=A_2+A_2+A_3+A_4+A_5+A_6+A_8.
$$
We denote by $|A_k|$
 the number of elements in
$A_k$.
Since $\mathbb F_p^\ast$ is cyclic, one has
$$
|A_k|=\frac{p-1}{\gcd(p-1,k)}\geq \frac{p-1}{k}.
$$
Set $B=A_2+A_2$. By Lemma \ref{Cauchy}, we get
$$
|B|=|A_2+A_2|\geq \min(p,\,2|A_2|-1)
      \geq \min\!\left(p,\,2\cdot\frac{p-1}{2}-1\right)=p-2.
$$
Then
\begin{align*}
|B+A_3|
   &\geq \min(p,\,|B|+|A_3|-1)\\
   &\geq \min(p,\,(p-2)+3-1)=p.
\end{align*}
Thus $B+A_3=\mathbb F_p$, and hence $S=\mathbb F_p$. The lemma follows.
\end{proof}

\begin{lem}\label{Lemma 3.5}
For $\alpha \in \mathfrak{M},$  we have
$$
g_k(\alpha)=\varphi(q)^{-1}S_k(q,a)\,v_k(\alpha-a/q)+O(P_k L^{-4B}) \quad (k=2,3,4,6,8)
$$
and
$$
g_5(a/q)=\varphi(q)^{-1}S_5(q,a)\,P_5^{\gamma}+O(P_5^{\gamma} L^{-4B}).
$$
\end{lem}
\begin{proof}
    It follows from similar arguments used in the proof of Lemma 6 in \cite{Hua}
\end{proof}

\begin{lem}\label{Lemma 3.6}
  For large odd $n,$ we have
    $$\int_{\mathfrak{M}}G(\alpha)e(-\alpha n)d\alpha \gg n^{\Gamma}.$$
\end{lem}
\begin{proof}
For $p\leq P_5^{\gamma}$ and
$|\beta|\leq L^B/n$,
\begin{align*}
    e(\beta p^5)&=1+2\pi i\int_{0}^{\beta p^5}e(t)\,dt
    =1+O(L^B n^{-1} P_5^{5\gamma})
    =1+O(L^B n^{\gamma-1}),
\end{align*}
 by Lemma \ref{Lemma 3.5}, we get
\begin{align}\label{bb}
   g_5(\alpha)&=g_5(a/q)+O(P_5^{\gamma} L^B n^{\gamma-1}) \notag \\
   &=\phi^{-1}(q)S_5(q,a)P_5^{\gamma}+O(P_5^{\gamma}L^{-4B}).
\end{align}
Hence, by \eqref{bb} together with Lemma~\ref{Lemma 3.5}, we have
\begin{align}\label{ee}
G(\alpha)=P_5^{\gamma}\varphi(q)^{-7}U(q,a)\,w(\alpha-a/q)+O(n^{1+\Gamma}L^{-4B}).
\end{align}
Since $\mathfrak{M}$ is a set of measure $O(L^{3B}/n)$, by \eqref{ee} we have
\begin{align}\label{2.3}
    \int_{\mathfrak{M}} \; G(\alpha)e(-\alpha n)\,d\alpha
&=
P_5^{\gamma}\sum_{q\leq L^B}\frac{A_n(q)}{\varphi(q)^7}
\int_{-L^B/n}^{L^B/n} w(\beta)e(-\beta n)\,d\beta \notag\\
&\quad  +O(n^{\Gamma}L^{-B}).
\end{align}
By Lemma \ref{Lemma 3.3}
 one has
$w(\beta)\ll n^{15/8}(1+n|\beta|)^{-2},$
and
\begin{align}\label{L}
&\int_{-L^B/n}^{L^B/n} w(\beta)e(-\beta n)\,d\beta
= \int_{-1/2}^{1/2} w(\beta)e(-\beta n)\,d\beta
   +O\bigg(\int_{L^B/n}^{1/2} |w(\beta)|d\beta\bigg) \notag\\
&= \bigg(\frac{1}{2^2\cdot3\cdot4\cdot6\cdot 8}\times\notag\\
&\sum_{\substack{
m_1+m_2+m_3+m_4+m_6+m_8=n\\
2^{-2}n<m_1,m_2\leq n\; \\
2^{-k}n<m_k\leq n\;(k=3,4,6,8)
}}
(m_1m_2)^{-1/2}m_3^{-2/3}m_4^{-3/4}m_6^{-5/6}m_8^{-7/8}\bigg)
\notag\\
&\quad + O(n^{7/8}L^{-B})\notag\\
& \gg
\sum_{\substack{m_1+m_2+m_3+m_4+m_6+m_8=n
\notag\\
2^{-2}n<m_1,m_2\leq \frac{5n}{16}\; \notag\\2^{-3}n<m_3\leq \frac{5n}{32}, \; 2^{-4}n<m_4\leq \frac{5n}{64}\notag\\
2^{-6}n<m_6\leq \frac{5n}{256}
}}
(m_1m_2)^{-1/2}m_3^{-2/3}m_4^{-3/4}m_6^{-5/6}m_8^{-7/8}\notag\\
&  \gg n^{7/8}.
\end{align}

On the other hand,
by Lemma \ref{Lemma 3.2} one   has
$A_n(q)\ll q^{9/2+\varepsilon}$ uniformly in $n,$
hence the series $$\mathfrak{S}(n)=\sum_{q=1}^{\infty}\varphi(q)^{-7}A_n(q) $$ converges absolutely, and
 \begin{align}\label{kimi}
     \sum_{q\leq L^B}\varphi(q)^{-7}A_n(q)=\mathfrak{S}(n)+O(L^{-B}).
 \end{align}
By Lemma~\ref{Lemma 3.1} and Lemma~\ref{Lemma 3.4}, together with the direct verification
that \(A_n(4)=0\), we have
$A_n(q)=0$
unless \(q\) is square-free.
 Then by Lemma \ref{Lemma 3.4} again we obtain
\begin{align}\label{S}
    \mathfrak{S}(n)
    =\prod_p\left(1+(p-1)^{-7}A_n(p)\right).
\end{align}
 By orthogonality of additive characters modulo $p$ and Lemma \ref{Mn} we see that
\begin{align}\label{kill}
    1+(p-1)^{-7}A_n(p)=p(p-1)^{-7}M_n(p)>0.
\end{align}
For $p>300,$ by Lemma \ref{Lemma 3.2} ($\mathrm{ii}$) we have
\begin{align*}
\left|(p-1)^{-7} A_n(p)\right|
&\leq \frac{(p-1)(\sqrt{p}+1)^2 (7\sqrt{p}+1)\prod\limits_{i=2}^5 \left(i\sqrt{p}+1\right)}{(p-1)^7} \\
&= \frac{p^{\frac{7}{2}}\left(1+\frac{1}{\sqrt{p}}\right)^2
\left(7+\frac{1}{\sqrt{p}}\right)
\prod\limits_{i=2}^5\left(i+\frac{1}{\sqrt{p}}\right)}{(p-1)^6} \\
& \! \leq \frac{2^6}{p^{\frac{5}{2}}}
\left(1+\frac{1}{\sqrt{p}}\right)^2
\left(2+\frac{1}{\sqrt{p}}\right)\\
& \quad \times  \left(3+\frac{1}{\sqrt{p}}\right)\left(4+\frac{1}{\sqrt{p}}\right)
\left(5+\frac{1}{\sqrt{p}}\right)\left(7+\frac{1}{\sqrt{p}}\right) \\
&\leq \frac{8\times10^4}{p^{\frac{5}{2}}}.
\end{align*}
So we get
\begin{align}\label{Y}
   \prod_{p>3 00}\left(1+(p-1)^{-7}A_n(p)\right)\geq \prod_{p>300}\left(1-\frac{8\times10^4}{p^{\frac{5}{2}}}\right)\geq c>0.
\end{align}
Then  by \eqref{S}--\eqref{Y} and Lemma \ref{Mn} one has
\begin{align}\label{2.9}
  \mathfrak{S}(n)
  &\geq c\prod_{p\leq 300}\left(1+(p-1)^{-7}A_n(p)\right)\notag \\
  &= c \prod_{p\leq 300}\left(p(p-1)^{-7}M_n(p)\right)\notag\\
&\geq c\prod_{p\leq 300} p^{-6}>0.
\end{align}

 Now by \eqref{2.3}--\eqref{kimi} and \eqref{2.9}, the lemma follows.

\end{proof}

\section{The Minor arcs }
Set
\begin{align*}
K_1&=\int_0^1|g_2(\alpha)g_3(\alpha)g_6(\alpha)|^2d\alpha, \\
K_2&=\int_0^1|g_2(\alpha)|^2|g_4(\alpha)|^4d\alpha,\\
K_3&=\int_0^1 |g_2(\alpha)g_4(\alpha)g_8^2(\alpha)|^2d\alpha.
\end{align*}
We begin by collecting several mean-value estimates for the exponential
sums appearing in \(G(\alpha)\). The following general result will be used
to estimate the auxiliary mean values \(K_1\), \(K_2\) and \(K_3\).
\begin{lem} \cite[Lemma 1] {Brudern} \label{Brudern}
For $j\geq 1,$ write $$f_j(\alpha)=\sum\limits_{x\leq n^{1/j}}e(\alpha x^j).$$
   Let $2\leq k_1\leq \cdots\leq k_s$ be natural numbers such that
    $$ \sum_{i=j+1}^s\frac{1}{k_i}\leq \frac{1}{k_j}, \quad 1\leq j\leq s-1 .$$
    Then we have
    $$ \int_{0}^1 \bigg| \prod_{i=1}^s f_{k_i}(\alpha)\bigg|^2\ll n^{\frac{1}{k_1}+\cdots\frac{1}{k_s}+\varepsilon}.   $$
\end{lem}
Applying Lemma \ref{Brudern} to the particular combinations of exponents occurring
in our argument gives the following estimates.
\begin{lem}\label{k_i}
    We have
    $$ K_i\ll n^{1+\varepsilon} \quad (i=1,2,3). $$
\end{lem}
\begin{proof}
   This follows immediately from Lemma \ref{Brudern}.
\end{proof}

We also require the means
\begin{align}\label{J_8}
J_8=\int_0^1|g_2(\alpha)g_8(\alpha)|^2|g_5(\alpha)|^4d\alpha, J_6=\int_0^1|g_2(\alpha)g_6(\alpha)|^2|g_5(\alpha)|^4d\alpha.
\end{align}
The estimates for \(J_8\) and \(J_6\) require more information than the
general mean-value estimate above. We therefore invoke an iterative estimate by Vaughan \cite{Vaughan1986}, which converts the relevant mean values into estimates for the number of solutions of certain auxiliary Diophantine equations.

\begin{lem} \cite [Lemma 4]{Vaughan1986} \label{Vaughan}
Suppose that $k\ge 3$, $1\ge \lambda \ge 1-1/k$, $v=k\lambda-k+1,$   and $P$ is a real number which is
sufficiently large in terms of $\varepsilon.$  Let $R(m)$ denote a non-negative integer-valued arithmetical function
with support in $[1,300P^{k\lambda}]$, let
$$
R_1(m)=\sum\limits_{\substack{P<x\leq 2P\\ \quad {x^k+n=m} }}\sum_{n} R(n),
$$
and
$$
S=\sum_m R(m)^2,\qquad T=\sum_m R_1(m)^2.
$$
Then for $1\leq j\leq k-2$ we have
$$
T \ll PS + P^{\,1+v-2^{\,1-j}}S
      + P^{\,1+\varepsilon+v(1-2^{-j})-(j+1)2^{\,-j}}
        S^{\,1-2^{-j}}\Bigl(\sum_m R(m)\Bigr)^{2^{\,1-j}}.
$$

\end{lem}

\begin{lem}\label{choice J_8}
We have
$$
J_8 \ll P_2^{1+\varepsilon}P_8P_5^{2\gamma}.
$$
\end{lem}
\begin{proof}
We first obtain an estimate for $T_1$,    the number of integral solutions of the equation
$$
y_1^8-y_2^8=z_1^5+z_2^5-z_3^5-z_4^5,
$$
where
\begin{align*}
    \qquad \frac12 P_8<y_i\leq P_8,\qquad 1\leq z_j\leq P_5^{\gamma}
\quad (1\leq i\leq 2,\ 1\leq j\leq 4).
\end{align*}

In Lemma  \ref{Vaughan}, we take $k=8$, $j=3$, $\lambda=\gamma$,
$P=\tfrac12 P_8$, $v=8\gamma-7$ and $R(m)$ as the number of choices for $z_1,z_2\in[1,P_5^{\gamma}]$ with
$z_1^5+z_2^5=m$. Then, the quantity $T$ in Lemma \ref{Vaughan} coincides with our $T_1$.

Observe that $z_1^5+z_2^5=m$ forces $(z_1+z_2)\mid m$. A standard divisor counting argument therefore yields
 $R(m)\ll m^{\varepsilon}$. Hence we have
$$
\sum_m R(m)\leq P_5^{2\gamma},\qquad
\sum_m R(m)^2 \ll P_5^{2\gamma+\varepsilon}.
$$
Inserting these estimates into Lemma \ref{Vaughan} we obtain
$$
T_1 \ll P_8 P_5^{2\gamma+\varepsilon}+P_8^{8\gamma-25/4}P_5^{2\gamma+\varepsilon}
+ P_8^{1/2+\frac{7}{8}(8\gamma-7)+\varepsilon}(P_5^{2\gamma})^{9/8}
\ll P_8^{1+\varepsilon}P_5^{2\gamma}.
$$
 Considering the underlying Diophantine equation, we have
 \begin{align*}
&\int_0^1 |g_2(\alpha) g_8(\alpha)|^2 |g_5(\alpha)|^4 \, d\alpha
\\
& \ll L^8
\sum_{\substack{\frac{P_8}{2}<y_1,y_2\leq P_8\\ 1<z_1,z_2,z_3,z_4\leq P_5^\gamma}}
\sum_{\substack{1<x_1,x_2\leq P_2\\
x_1^2-x_2^2=y_1^8-y_2^8+z_1^5+z_2^5-z_3^5-z_4^5}} 1 \\
&\ll L^8 P_2
\sum_{\substack{\frac{P_8}{2}<y_1,y_2\leq P_8\\ 1<z_1,z_2,z_3,z_4\leq P_5^\gamma\\
y_1^8-y_2^8+z_1^5+z_2^5-z_3^5-z_4^5=0}} 1 \\
&\quad{}+ L^8
\sum_{\substack{\frac{P_8}{2}<y_1,y_2\leq P_8\\ 1<z_1,z_2,z_3,z_4\leq P_5^\gamma\\
y_1^8-y_2^8+z_1^5+z_2^5-z_3^5-z_4^5\ne0}}
\tau\!\left(\left|y_1^8-y_2^8+z_1^5+z_2^5-z_3^5-z_4^5\right|\right) \\
&\ll L^8 P_2T_1 + L^8 P_2^{\varepsilon} P_8^2P_5^{4\gamma} \\
& \ll P_2P_8^{1+\varepsilon}P_5^{2\gamma}+P_2^{\varepsilon}P_8^2P_5^{4\gamma}\\
&\ll P_2^{1+\varepsilon}P_8P_5^{2\gamma},
\end{align*}
where $\tau(n)$ is the divisor function.
Therefore, the proof of the lemma is complete.
\end{proof}
A similar argument is needed for $J_6$. Since the choice of parameters in Lemma \ref{Vaughan} is different here, we include the details for completeness.
    \begin{lem} \label{choice J_6}
We have
$$
J_6 \ll P_6^{2+\varepsilon}P_5^{4\gamma}.
$$
\end{lem}
\begin{proof}
We first get an estimate for $T_2$,    the number of integral solutions of the equation
$$
y_1^6-y_2^6=z_1^5+z_2^5-z_3^5-z_4^5,
$$
where
\begin{align*}
    \qquad \frac12 P_6<y_i\leq P_6,\qquad 1\leq z_j\leq P_5^{\gamma}
\quad (1\leq i\leq 2,\ 1\leq j\leq 4).
\end{align*}
In Lemma \ref{Vaughan}, we specialize to the parameters \(k=6\), \(j=2\), \(\lambda=\gamma\), \(P=\tfrac12 P_6\), \(v=6\lambda-5\), and define \(R(m)\) as the number of choices for \(z_1,z_2\in[1,P_5^{\gamma}]\) with \(z_1^5+z_2^5=m\). Thus the quantity \(T\) in Lemma \ref{Vaughan} coincides with our \(T_2\).

As \(z_1^5+z_2^5=m\), we deduce \((z_1+z_2)\mid m\). A standard divisor counting argument thus gives \(R(m)\ll m^{\varepsilon}\). Hence we have
$$
T_2 \ll P_6 P_5^{2\gamma+\varepsilon}
+P_6^{6\gamma-9/2}P_5^{2\gamma}+ P_6^{(9\gamma-7)/2}(P_5^{5\gamma /2})
\ll P_6^{(9\gamma-7)/2}P_5^{5\gamma /2}.
$$
 Considering the underlying Diophantine equation, we have
 \begin{align*}
&\int_0^1 |g_2(\alpha) g_6(\alpha)|^2 |g_5(\alpha)|^4 \, d\alpha
\ll L^8P_2T_2 +L^8P_2^{\varepsilon}P_6^2P_5^{4\gamma}
 \ll P_6^{2+\varepsilon}P_5^{4\gamma}.
\end{align*}
Therefore, the proof of the lemma is complete.
\end{proof}

We now estimate the contribution from the set \(\mathfrak n\). Combining the preceding mean-value estimates with H\"older's inequality gives a power saving.
\begin{lem} \label{n}
We have
    $$\int_\mathfrak{n} |G(\alpha)|\,d\alpha \ll n^{\Gamma-\frac{2}{8325}+\varepsilon}.$$
\end{lem}
\begin{proof}
 By H\"older's inequality and Lemma \ref{k_i}, together with Lemma \ref{choice J_8} and Lemma \ref{choice J_6}, we have
\begin{align*}
   \int_{ \mathfrak{n}}|G(\alpha)|\,d\alpha &\ll J_6^{\frac{1}{12}}J_8^{\frac{1}{6}}K_1^{\frac{5}{12}}K_2^{\frac{1}{6}}K_3^{\frac{1}{6}}
\sup_{\alpha\in \mathfrak{n}} |g_3(\alpha)|^{1/6}\\
&\ll n^{\Gamma-\frac{2}{8325}+\varepsilon}.
\end{align*}
\end{proof}

We now recall two lemmas on exponential sums over primes, following Kumchev \cite{Kumchev}.
\begin{lem}\cite[Theorem 3]{Kumchev} \label{cubic}
Let
\[
h_3(\alpha;X)
=
\sum_{X<p\leq 2X} e(\alpha p^3),
\]
where \(X\) is sufficiently large and \(\alpha\in\mathbb R\). For every
\(\varepsilon>0\), there exist integers \(a\) and \(q\) satisfying
\[
1\leq q\leq X^{12/7},\qquad
(a,q)=1,\qquad
|q\alpha-a|\leq X^{-12/7},
\]
such that
\[
h_3(\alpha;X)
\ll
X^{13/14+\varepsilon}
+
\frac{X^{1+\varepsilon}}
{\bigl(q+X^3|q\alpha-a|\bigr)^{1/2}}.
\]
The implied constant depends at most on \(\varepsilon\).
\end{lem}

\begin{lem}\cite[Theorem 2]{Kumchev}\label{Lemma 5.2}
Let $k\in\mathbb{N},\alpha\in\mathbb{R}$ and $h_k(\alpha,X)=\sum_{X<p\leq 2X}e(\alpha p^k).$   Suppose that there exist
$a\in\mathbb{Z}$ and $q\in\mathbb{N}$ satisfying
$$
1\leq q\leq Q,\qquad (a,q)=1,\qquad |q\alpha-a|<QX^{-k}.
$$
with $Q\leq X$. Then, for any fixed $\varepsilon>0$,
$$
h_k(\alpha,X)\ll Q^{1/2}X^{11/20+\varepsilon}
+\frac{q^{\varepsilon}X(\log X)^C}{(q+X^k|q\alpha-a|)^{1/2}},
$$
where $C>0$ is an absolute constant and the implied constant in $\ll$ depends
at most on $k$ and $\varepsilon$.
\end{lem}

We now estimate the contribution from $\widetilde{\mathfrak m}$.
The following estimate yields the required logarithmic saving.
\begin{lem}\label{pruning}
We have
\[
\int_{\widetilde{\mathfrak m}}|G(\alpha)|\,d\alpha
\ll n^\Gamma L^{-1}.
\]
\end{lem}

\begin{proof}
For $\alpha\in\widetilde{\mathfrak m}$, put
\[
V(\alpha)=\bigl(q+n|q\alpha-a|\bigr)^{-1}.
\]
Since
\[
q\leq Q_1,\qquad |q\alpha-a|\leq Q_2^{-1},
\qquad Q_1Q_2=n,
\]
we have
\[
q+n|q\alpha-a|\leq 2Q_1,
\qquad
V(\alpha)\gg Q_1^{-1}.
\]
Applying Lemma \ref{Lemma 5.2} with \(X=P_k/2\) and \(Q=Q_1\), followed by
partial summation, we obtain
\[
g_k(\alpha)
\ll
Q_1^{1/2}P_k^{11/20+\varepsilon}
+
P_kL^{C+1}V(\alpha)^{1/2-\varepsilon}
\qquad (k=2,3,4,6,8).
\]
Since $Q_1=n^{1/25}$ and $P_k\geq P_8=n^{1/8}$, we have
\[
Q_1^{1/2}P_k^{11/20+\varepsilon}
\ll
P_kQ_1^{-1/2+\varepsilon}
\ll
P_kL^{C+1}V(\alpha)^{1/2-\varepsilon}.
\]
Consequently,
\[
g_k(\alpha)
\ll
P_kL^{C+1}V(\alpha)^{1/2-\varepsilon}
\qquad (k=2,3,4,6,8).
\]

Together with the trivial estimate
\[
g_5(\alpha)\ll P_5^\gamma L,
\]
this yields
\[
G(\alpha)
\ll
n^{1+\Gamma}L^{6C+7}V(\alpha)^{3-6\varepsilon}
\ll
n^{1+\Gamma}L^{6C+7}V(\alpha)^{5/2}.
\]

Writing
\[
\alpha=\frac aq+\beta,
\qquad
V(\alpha)=q^{-1}(1+n|\beta|)^{-1},
\]
we obtain
\[
\begin{aligned}
\int_{\mathfrak m_2}|G(\alpha)|\,d\alpha
&\ll
n^{1+\Gamma}L^{6C+7}
\sum_{q\leq L^B}q^{-3/2}
\int_{L^B/n}^{\infty}
(1+n\beta)^{-5/2}\,d\beta\\
&\ll
n^\Gamma L^{6C+7-3B/2}.
\end{aligned}
\]
Similarly,
\[
\begin{aligned}
\int_{\mathfrak m_3}|G(\alpha)|\,d\alpha
&\ll
n^{1+\Gamma}L^{6C+7}
\sum_{L^B<q\leq Q_1}q^{-3/2}
\int_0^{1/(qQ_2)}
(1+n\beta)^{-5/2}\,d\beta\\
&\ll
n^\Gamma L^{6C+7-B/2}.
\end{aligned}
\]
Taking $B=12C+18$, we conclude that
\[
\int_{\widetilde{\mathfrak m}}|G(\alpha)|\,d\alpha
\ll n^\Gamma L^{-1}.
\]
\end{proof}

\section{The proof of Theorem 1.1}

We first prove that
\[
\left[0,1\right]\setminus(\mathfrak M\cup\mathfrak n)
\subseteq \widetilde{\mathfrak m}.
\]
Let
\[
\alpha\in \left[0,1\right] \setminus(\mathfrak M\cup\mathfrak n).
\]
Then
\[
|g_3(\alpha)|>P_3^{1-\eta}.
\]
Applying Lemma~\ref{cubic} and partial
summation, we have
\[
g_3(\alpha)
\ll
P_3^{13/14+\varepsilon}
+
\frac{P_3^{1+\varepsilon}}
{\bigl(q+P_3^3|q\alpha-a|\bigr)^{1/2}}.
\]
The first term is smaller than \(P_3^{1-\eta+\varepsilon}\), and hence
the second term must satisfy
\[
\frac{P_3^{1+\varepsilon}}
{\bigl(q+P_3^3|q\alpha-a|\bigr)^{1/2}}
\gg P_3^{1-\eta+\varepsilon}.
\]
Therefore,
\[
q+P_3^3|q\alpha-a|
\ll P_3^{2\eta+\varepsilon},
\]
and consequently
\[
q\ll n^{2\eta/3+\varepsilon},
\qquad
|q\alpha-a|\ll n^{-1+2\eta/3+\varepsilon}.
\]
Since \(\eta=11/200\), for sufficiently small \(\varepsilon>0\) and
sufficiently large \(n\), we obtain
\[
q\leq Q_1,
\qquad
|q\alpha-a|\leq Q_2^{-1}.
\]

If \(L^B<q\leq Q_1\), then \(\alpha\in\mathfrak m_3\). If
\(q\leq L^B\), then \(\alpha\notin\mathfrak M\) gives
\[
\left|\alpha-\frac aq\right|>\frac{L^B}{n},
\]
so that \(\alpha\in\mathfrak m_2\). Hence
\[
\left[0,1\right] \setminus(\mathfrak M\cup\mathfrak n)
\subseteq\widetilde{\mathfrak m}.
\]

Finally,
\[
\begin{aligned}
\left[0,1\right]\setminus\mathfrak M
&=
\bigl((\left[0,1\right]\setminus\mathfrak M)\cap\mathfrak n\bigr)
\cup
\bigl((\left[0,1\right]\setminus\mathfrak M)\setminus\mathfrak n\bigr)\\
&=
\bigl((\left[0,1\right]\setminus\mathfrak M)\cap\mathfrak n\bigr)
\cup
\bigl(\left[0,1\right]\setminus(\mathfrak M\cup\mathfrak n)\bigr)\\
&\subseteq
\mathfrak n\cup\widetilde{\mathfrak m}.
\end{aligned}
\]
Thus, it follows from Lemma \ref{Lemma 3.6}, Lemma \ref{n} and Lemma \ref{pruning} that
\begin{align}\label{5.1}
\int_0^1G(\alpha)e(-\alpha n)d\alpha
    & =\int_{\mathfrak{M}}G(\alpha)e(-\alpha n)d\alpha \notag\\
    &\quad +O\bigg(\int_\mathfrak{n} |G(\alpha)|\,d\alpha\bigg)+O\bigg(\int_{\widetilde{\mathfrak{m}}}|G(\alpha)|d\alpha \bigg) \notag \\
    &\gg  n^{\Gamma}.
\end{align}
Using \eqref{2.1} and \eqref{5.1}, we complete the proof of Theorem \ref{wang}.

\end{document}